\documentclass{amsart}
\usepackage{amsthm,amssymb,mathtools}
\usepackage[british]{babel}
\usepackage{enumitem}
\usepackage[latin1]{inputenc}
\usepackage{bussproofs}
\usepackage{url}
\usepackage{hyperref}
\usepackage[hang,flushmargin]{footmisc}

\newtheorem{theorem}{Theorem}
\numberwithin{theorem}{section}
\newtheorem{corollary}[theorem]{Corollary}
\newtheorem{lemma}[theorem]{Lemma}
\newtheorem{proposition}[theorem]{Proposition}
{\theoremstyle{definition}
\newtheorem{definition}[theorem]{Definition}
\newtheorem{remark}[theorem]{Remark}

}

\newcommand{\isigma}{\mathbf{I\Sigma}}

\newcommand{\pa}{\mathbf{PA}}
\newcommand{\ph}{\operatorname{PH}}
\newcommand{\pr}{\operatorname{Pr}}
\newcommand{\pro}{\operatorname{Proof}}
\newcommand{\con}{\operatorname{Con}}
\newcommand{\feps}{F_{\varepsilon_0}}
\newcommand{\true}{\operatorname{True}}

\newcommand{\drfn}{\operatorname{RFN}_{\Sigma_1}^\diamond(\pa)}
\newcommand{\rfn}{\operatorname{RFN}_{\Sigma_1}}

\newcommand{\dn}{\operatorname{N^\diamond}}
\newcommand{\n}{\operatorname{N}}

\makeatletter
\newcommand{\dotminus}{\mathbin{\text{\@dotminus}}}
\newcommand{\@dotminus}{%
  \ooalign{\hidewidth\raise1ex\hbox{.}\hidewidth\cr$\m@th-$\cr}%
}
\makeatother

\title[Proof Lengths for Instances of the Paris-Harrington Principle]{Proof Lengths for Instances of the Paris-Harrington Principle\footnotemark[1]}
\author{Anton Freund}

\begin{document}

\begin{abstract}
As Paris and Harrington have famously shown, Peano Arithmetic does not prove that for all numbers $k,m,n$ there is an $N$ which satisfies the statement $\ph(k,m,n,N)$: For any $k$-colouring of its $n$-element subsets the set $\{0,\dots,N-1\}$ has a large homogeneous subset of size $\geq m$. At the same time very weak theories can establish the $\Sigma_1$-statement $\exists_N\ph(\overline k,\overline m,\overline n,N)$ for any fixed parameters $k,m,n$. Which theory, then, does it take to formalize \emph{natural} proofs of these instances? It is known that $\forall_m\exists_N\ph(\overline k,m,\overline n,N)$ has a natural and short proof (relative to $n$ and $k$) by $\Sigma_{n-1}$-induction. In contrast, we show that there is an elementary function $e$ such that any proof of $\exists_N\ph (\overline{e(n)},\overline{n+1},\overline n,N)$ by $\Sigma_{n-2}$-induction is ridiculously long.\\
In order to establish this result on proof lengths we give a computational analysis of slow provability, a notion introduced by Sy-David Friedman, Rathjen and Weiermann. We will see that slow uniform $\Sigma_1$-reflection is related to a function that has a considerably lower growth rate than $\feps$ but dominates all functions $F_\alpha$ with $\alpha<\varepsilon_0$ in the fast-growing hierarchy.
\end{abstract}

\maketitle
{\let\thefootnote\relax\footnotetext{\copyright~2017. This manuscript version is made available under the CC-BY-NC-ND 4.0 license \url{http://creativecommons.org/licenses/by-nc-nd/4.0/}. The paper has been accepted for publication in the Annals of Pure and Applied Logic (doi 10.1016/j.apal.2017.01.004).}}

We recall some terminology from \cite{paris-harrington}: For a set $X$ and a natural number $n$ we write $[X]^n$ for the collection of subsets of $X$ with precisely $n$ elements. Given a function $f$ with domain $[X]^n$, a subset $Y$ of $X$ is called homogeneous for $f$ if the restriction of $f$ to the set $[Y]^n$ is constant. A non-empty subset of $\mathbb N$ is called large if its cardinality is at least as big as its minimal element. Where the context suggests it we use $N$ to denote the set $\{0,\dots ,N-1\}$. Then the Paris-Harrington Principle, or Strengthened Finite Ramsey Theorem, expresses that for all natural numbers $k,m,n$ there is an $N$ such that the following statement holds:
\begin{equation*}
\ph (k,m,n,N)\quad:\equiv\quad\parbox{8cm}{``for any function $[N]^n\rightarrow k$ the set $N$ has a large homogeneous subset with at least $m$ elements"}
\end{equation*}
Using the methods presented in \cite[Section I.1(b)]{hajek91} it is easy to formalize the statement $\ph (k,m,n,N)$ in the language of first order arithmetic, as a formula that is $\Delta_1$ in the theory $\isigma_1$ of $\Sigma_1$-induction. The celebrated result of \cite{paris-harrington} says that the formula $\forall_{k,m,n}\exists_N\ph (k,m,n,N)$ is true but unprovable in Peano Arithmetic.\\
As is well-known, any true $\Sigma_1$-formula in the language of first-order arithmetic can be proved in a theory as weak as Robinson Arithmetic. It is thus pointless to ask whether a $\Sigma_1$-sentence is provable in a sound arithmetical theory, in contrast to the situation for $\Pi_1$-sentences (cf.\ G\"odel's Theorems) and $\Pi_2$-sentences (provably total functions). What we can sensibly ask is whether a $\Sigma_1$-sentence has a proof with some additional property. The present paper explores this question for instances $\exists_N\ph (\overline k,\overline m,\overline n,N)$ of the Paris-Harrington Principle. Our principal result states that, for some elementary function $e$, the following holds:
\begin{equation}\label{eq:no-short-proofs-ph}
 \parbox{11cm}{For sufficiently large $n$, no proof of the formula $\exists_N\ph (\overline{e(n)},\overline{n+1},\overline n,N)$ in the theory $\isigma_{n-2}$ can have G\"odel number smaller than $\feps(n-3)$.}
\end{equation}
If we replace $\isigma_{n-2}$ by $\isigma_{n-3}$ (and $\feps(n-3)$ by $\feps(n-4)$) then we can take the constant function $e(n)=8$. It is open whether we can make $e$ constant and keep the stronger fragment $\isigma_{n-2}$.\\
Recall that $\feps$ is the function at stage $\varepsilon_0$ of the fast-growing hierarchy. Ketonen and Solovay in \cite{solovay81} have related it to the function that maps $(k,m,n)$ to the smallest witness $N$ which makes the statement $\ph (\overline k,\overline m,\overline n,\overline N)$ true. A classical result due to Kreisel, Wainer and Schwichtenberg \cite{kreisel52,wainer70,schwichtenberg71} says that $\feps$ eventually dominates any provably total function of Peano Arithmetic. Similar to (\ref{eq:no-short-proofs-ph}) we will show that the $\Sigma_1$-formula $\exists_y\,\feps(\overline n)=y$ has no short proof in the theory $\isigma_n$.\\
By \cite[Theorem II.1.9]{hajek91} the formula $\forall_m\exists_N\ph (\overline k,m,\overline n,N)$ is provable in $\isigma_{n-1}$, for each fixed $n\geq 2$ and $k$. The proofs of these instances formalize perfectly natural mathematical arguments. According to \cite[Section II.2(c)]{hajek91} they can be constructed in the meta-theory $\isigma_1$. Since all provably total functions of $\isigma_1$ are primitive recursive, this complements (\ref{eq:no-short-proofs-ph}) by the following statement:
\begin{equation}\label{eq:short-proofs-ph}
 \parbox{11cm}{There is a primitive recursive function which maps $(k,n)$ with $n\geq 2$ to a proof of the formula $\forall_m\exists_N\ph (\overline k,m,\overline n,N)$ in the theory $\isigma_{n-1}$.}
\end{equation}
Similarly, a primitive recursive construction yields proofs of $\exists_y\,\feps(\overline n)=y$ in the theories $\isigma_{n+1}$: In view of $\feps(x)\simeq F_{\omega_{x+1}}(x)=F_{\omega_x^{x+1}}(x)$ it suffices to prove the statements ``$F_{\omega_n^{n+1}}$ is total''. This is done by $\Pi_2$-induction up to $\omega_n^{n+1}$, which is available in $\isigma_{n+1}$ by Gentzen's classical construction (cf.\ \cite[Theorem 4.11]{wainer-fairtlough-98}).\\

We argue that (\ref{eq:no-short-proofs-ph}) is not only a result about proof length, but also about the existence of natural proofs: Observe first that we are concerned with sequences $p_n$ of proofs for a sequence of parametrized statements $A_n$, rather than with a single proof of a single statement. Under which conditions can such a sequence of proofs follow an intelligible uniform proof idea? It is the role of the proofs $p_n$ to guarantee that the formulas $A_n$ are true. On the other hand the statement ``the given proof idea leads to formally correct proofs $p_n$ of the statements $A_n$'' should, we believe, be justified by fairly elementary means. Since elementary means cannot prove the totality of functions with a high growth rate this implies that the function mapping $n$ to (a code of) the proof $p_n$ cannot grow too fast. In this sense (\ref{eq:no-short-proofs-ph}) shows that $\isigma_{n-2}$-proofs of the Paris-Harrington Principle for arity $n$ and $e(n)$ colours cannot follow a natural proof idea. The author sees no formal condition which would, on the positive side, ensure that a sequence of proofs is natural. On an informal level the construction which establishes \cite[Theorem II.1.9]{hajek91} appears to provide natural $\isigma_{n-1}$-proofs of the statements $\forall_m\exists_N\ph (\overline k,m,\overline n,N)$.\\

Let us briefly discuss connections with a line of research initiated by Harvey Friedman: Theorem 15 in \cite{friedman-proof-length} says that any proof of a certain $\Sigma^0_1$-statement in the theory $\Pi^1_2\text{-BI}_0$ must have at least $2_{1000}$ (i.e.\ $1000$ iterated exponentials to the base $2$) symbols. Obviously this goes much further than our result insofar as it involves a much stronger theory. However, there is also a more conceptual difference: Friedman's statement can, in principle, be verified explicitly (by looking at all possible proofs with less than $2_{1000}$ symbols) and is thus finitistically meaningful. In contrast, our statement (\ref{eq:no-short-proofs-ph}) involves an unbounded existential quantifier, implicit in the phrase ``sufficiently large". It is conceivable that any witness to this existential quantifier is so huge that statement (\ref{eq:no-short-proofs-ph}) does not have ``practical significance". On the other hand the more abstract form of (\ref{eq:no-short-proofs-ph}) has the important advantage of making the statement more robust: A result like \cite[Theorem 15]{friedman-proof-length} requires concrete numerical bounds which might depend on the formalization and are difficult to establish in full detail. To prove claim (\ref{eq:no-short-proofs-ph}), on the other hand, we can rely on the more robust concept of growth rates. How exactly we arithmetize the relation ``$p$ codes a proof of the statement with G\"odel number $\varphi$ in the theory $\isigma_n$" will not matter. All we require is that this relation is defined by an arithmetical formula $\pro_{\isigma_n}(p,\varphi)$ (with parameters $n,p$ and $\varphi$) which is $\Delta_1$ in the theory $\isigma_1$ and $\isigma_1$-provably equivalent to the usual formalizations of provability. Statement (\ref{eq:no-short-proofs-ph}) is true for any such arithmetization; merely the concrete meaning of ``sufficiently large" may change (cf.\ Remark \ref{rmk:preprocess-proofs} below). Another interesting comparison is with a result of Kraj\'i\v{c}ek \cite[Theorem 6.1]{krajicek89}: He considers $\Pi_2$-instances of the Paris-Harrington Principle and establishes linear bounds on the number of steps in proofs in full Peano Arithmetic (rather than in restricted fragments).\\

To conclude this introduction, let us summarize the different sections of the paper: In Section \ref{sect:bounding-proof-sizes} we show how the analysis of reflection leads to lower bounds on proof sizes. Given a theory $\mathbf T$, the uniform reflection principle for the formula $\exists_y\varphi(x,y)$ expresses that ``for all $p$ and $n$ there is an $N$ such that if $p$ is a $\mathbf T$-proof of $\exists_y\varphi(\overline n,y)$ then $\varphi(\overline n,\overline N)$ is true''. If we have a bound on the provably total functions of reflection then we know that the witness $N$ cannot be too much bigger than the code of the proof $p$. Vice versa $p$ cannot be too small if $\exists_y\varphi(\overline n,y)$ has only large witnesses. We suppose that this line of argumentation is known (it occurs e.g.\ in \cite{pudlak93}), but the author knows of no article that would develop it in general form.\\
The method just described applies to sequences of proofs in a single theory $\mathbf T$, while statement (\ref{eq:no-short-proofs-ph}) is concerned with a sequence of proofs that may contain axioms from increasingly strong theories. This discrepancy is resolved in Section \ref{sect:no-short-proofs-ph}: We consider a notion of ``slow proof'' in Peano Arithmetic, deduced from the slow consistency statement introduced by Sy-David Friedman, Rathjen and Weiermann in \cite{rathjen13}. The idea is to penalize complex induction axioms by a drastic increase in proof size. This generates an interplay between proof length and the use of induction. At the same time it makes the construction of proofs more difficult, thus weakening the reflection and consistency statement. We can then apply the method of Section \ref{sect:bounding-proof-sizes} to show that any \emph{slow} $\pa$-proof of $\exists_N\ph (\overline{e(n+2)},\overline{n+3},\overline{n+2},N)$ must be long. Claim (\ref{eq:no-short-proofs-ph}) will easily follow.\\
The results of Section \ref{sect:no-short-proofs-ph} rely on certain bounds on the provably total functions of slow reflection. The proof of these bounds follows in Section \ref{sect:provably-total-slow-reflection}. There we relate slow uniform $\Sigma_1$-reflection to a ``slow variant" $\feps^\diamond$ of the function $\feps$. We will see that each function $F_\alpha$ with $\alpha<\varepsilon_0$ is dominated by $\feps^\diamond$ while $\feps^\diamond$ itself grows much slower than $\feps$. This computational analysis of slow reflection is complemented by the results of \cite{freund-slow-reflection}, where we investigate the consistency strength ($\Pi_1$-consequences) of slow reflection. Further results on slow provability can be found in \cite{henk-pakhomov}.

\section{Bounding Proof Sizes via Reflection Principles}\label{sect:bounding-proof-sizes}

In this section we show how bounds on the provably total functions of uniform $\Sigma_1$-reflection lead to lower bounds on the sizes of proofs. To formulate the reflection principle we will need a $\Sigma_1$-formula $\true_{\Sigma_1}(\varphi)$ that defines truth for $\Sigma_1$-formulas (in the large sense, i.e.\ the formula may start with several existential quantifiers). The theory $\isigma_1$ should be able to prove Tarski's truth conditions (as guaranteed by \cite[Theorem I.1.75]{hajek91}). With respect to the proof predicate we must develop the theory in some generality:

\begin{definition}\label{def:proof-predicate-general}
 A proof predicate is a $\Pi_1$-formula $\pro(p,\varphi)$ in the language of first-order arithmetic, with only the variables $p$ and $\varphi$ free. Given a proof predicate we have the associated $\Sigma_1$-reflection principle
\begin{equation*}
 \rfn :\equiv \forall_\varphi(\text{``$\varphi$ is a closed $\Sigma_1$-formula''}\land\exists_p\pro(p,\varphi)\rightarrow\true_{\Sigma_1}(\varphi)).
\end{equation*}
For a natural number $p$ and a formula $\varphi$ with Gödel number $\ulcorner\varphi\urcorner$ we say that ``$p$ is a proof of $\varphi$" if the formula $\pro(\overline p,\overline{\ulcorner\varphi\urcorner})$ is true in the standard model.
\end{definition}

The following observation is easy but crucial:

\begin{lemma}\label{lem:proof-size-computes-witness}
 Let $\pro(p,\varphi)$ be a proof predicate, and let $\mathbf T$ be a sound extension of $\isigma_1$ that proves the $\Sigma_1$-reflection principle associated with $\pro(p,\varphi)$. For any $\Sigma_1$-formula $\psi(x,y)$ there is a $\mathbf T$-provably total function $g:\mathbb N^2\rightarrow\mathbb N$ such that $\psi(\overline n,\overline{g(p,n)})$ is true whenever $p$ is a proof of $\exists_y\psi(\overline n,y)$.
\end{lemma}

Note that, since $\mathbf T$ must be sound, the lemma can only be applied to proof predicates which are themselves sound for $\Sigma_1$-formulas.

\begin{proof}
Since the theory $\mathbf T$ extends $\isigma_1$ it is strong enough to handle Feferman's dot notation, and it proves the ``It's snowing"-Lemma (see \cite[Corollary I.1.76]{hajek91}). Combining this with the reflection principle for $\pro(p,x)$ we obtain
\begin{equation*}
 \mathbf T\vdash\forall_x(\exists_p\pro(p,\ulcorner\exists_y\psi(\dot x,y)\urcorner)\rightarrow\exists_y\psi(x,y)).
\end{equation*}
Prefixing quantifiers transforms this into
\begin{equation*}
 \mathbf T\vdash\forall_{p,x}\exists_y(\pro(p,\ulcorner\exists_y\psi(\dot x,y)\urcorner)\rightarrow\psi(x,y)).
\end{equation*}
We remark that it is only mildly non-constructive to prefix the existential quantifier in the consequent: A computation of the witness $y$ will use the proof $p$ but rather not the computational content of the statement $\pro(p,\ulcorner\exists_y\psi(\dot x,y)\urcorner)$. In any case the formula $\pro(p,\ulcorner\exists_y\psi(\dot x,y)\urcorner)\rightarrow\psi(x,y)$ is $\Sigma_1$ in $\isigma_1$. As we have seen the theory $\mathbf T$ shows that this formula defines a left-total relation. To obtain a single-valued function we apply a standard minimization argument. Note that we cannot simply pick the minimal value for $y$ since this would yield a function with a $\Delta_2$-graph; instead we simultaneously minimize over $y$ and the witness to the existential quantifier implicit in $\pro(p,\ulcorner\exists_y\psi(\dot x,y)\urcorner)\rightarrow\psi(x,y)$. This results in a $\Sigma_1$-formula $\chi(p,x,y)$ such that we have
\begin{equation*}
 \mathbf T\vdash\forall_{p,x,y}(\chi(p,x,y)\rightarrow(\pro(p,\ulcorner\exists_y\psi(\dot x,y)\urcorner)\rightarrow\psi(x,y)))
\end{equation*}
and $\mathbf T\vdash\forall_{x,p}\exists!_y\chi(p,x,y)$. Since $\mathbf T$ is sound the formula $\chi(p,x,y)$ does indeed define a $\mathbf T$-provably total function $g:\mathbb N^2\rightarrow\mathbb N$, which satisfies $\mathbb N\vDash\chi(\overline p,\overline n,\overline{g(p,n)})$ for all natural numbers $p$ and $n$. By the above we also have
\begin{equation*}
 \mathbb N\vDash\pro(\overline p,\overline{\ulcorner\exists_y\psi(\overline n,y)\urcorner})\rightarrow\psi(\overline n,\overline{g(p,n)})\qquad\text{for all $p,n\in\mathbb N$}.
\end{equation*}
Lifting the implication to the meta-language gives the desired claim.
\end{proof}

We can deduce the promised lower bound on proof sizes:

\begin{proposition}\label{prop:lower-bounds-proof-general}
Let $\pro(p,\varphi)$ be a proof predicate, and let $\mathbf T$ be a sound extension of $\isigma_1$ that proves the $\Sigma_1$-reflection principle for $\pro(p,\varphi)$.  Consider a $\Sigma_1$-formula $\psi(x,y)$ and define a function $F_\psi:\mathbb N\rightarrow\mathbb N\cup\{\infty\}$ by setting
\begin{equation*}
F_\psi(n):=\begin{cases}
m\quad &\text{if $m$ is the least number for which $\psi(\overline n,\overline m)$ is true},\\
\infty\quad &\text{if $\exists_y\psi(\overline n,y)$ is false}.
\end{cases}
\end{equation*}
Let $f:\mathbb N\rightarrow\mathbb N$ be a function with $f(n)\geq n$ and such that, whenever $g$ is $\mathbf T$-provably total, the function $g\circ f$ is eventually dominated by $F_\psi$ (considering $\infty$ as bigger than any natural number). Then there is a bound $N$ such that we have
\begin{equation*}
p>f(n)\qquad\text{whenever $p$ is a proof of $\exists_y\psi(\overline n,y)$ with $n\geq N$}.
\end{equation*}
\end{proposition}

To avoid misunderstanding, we stress that the notion of proof in the last line of the proposition is induced by the proof predicate in the first line, via Definition \ref{def:proof-predicate-general}.

\begin{proof}
Let $g:\mathbb N^2\rightarrow\mathbb N$ be the function provided by Lemma \ref{lem:proof-size-computes-witness}. We can make $g$ monotone in both arguments: First define $g_0:\mathbb N^2\rightarrow\mathbb N$ by the primitive recursion
\begin{align*}
g_0(p,0)&:=g(p,0),\\
g_0(p,n+1)&:=\max\{g(p,n+1),g_0(p,n)\}.
\end{align*}
This yields $g_0(p,n)\geq g(p,n)$ for all numbers $p$ and $n$, as well as $g_0(p,n)\leq g_0(p,n')$ whenever we have $n\leq n'$. Now define $g_1:\mathbb N^2\rightarrow\mathbb N$ by setting
\begin{align*}
g_1(0,n)&:=g_0(0,n),\\
g_1(p+1,n)&:=\max\{g_0(p+1,n),g_1(p,n)\}.
\end{align*}
It is obvious that we have $g_1(p,n)\geq g_0(p,n)\geq g(p,n)$ for all numbers $p$ and $n$, and that $p\leq p'$ implies $g_1(p,n)\leq g_1(p',n)$. By induction on $p$ one can also show that $g_1(p,n)\leq g_1(p,n')$ holds whenever we have $n\leq n'$. Lemma \ref{lem:proof-size-computes-witness} implies that we have
\begin{equation*}
F_\psi(n)\leq g_1(p,n)\qquad\text{whenever $p$ is a proof of $\exists_y\psi(\overline n,y)$}.
\end{equation*}
Since the theory $\mathbf T$ extends $\isigma_1$ its provably total functions are closed under primitive recursion, by \cite[Theorem I.1.54]{hajek91}. Thus $g_1$ is still $\mathbf T$-provably total. We define another $\mathbf T$-provably total function $g^\triangle:\mathbb N\rightarrow\mathbb N$, diagonalizing over $g_1$, as
\begin{equation*}
g^\triangle(p):=g_1(p,p)+1.
\end{equation*}
By assumption there is a bound $N$ such that we have
\begin{equation*}
(g^\triangle\circ f)(n)\leq F_\psi(n)\qquad\text{for all $n\geq N$}.
\end{equation*}
Let us show that the same bound $N$ satisfies the claim of the proposition: Consider an arbitrary $n\geq N$ and assume that $p$ is a proof of the formula $\exists_y\psi(\overline n,y)$. Aiming at a contradiction we assume $p\leq f(n)$. Then we have
\begin{equation*}
F_\psi(n)\leq g_1(p,n)\leq g_1(f(n),f(n))<(g^\triangle\circ f)(n)\leq F_\psi(n),
\end{equation*}
which is indeed absurd.
\end{proof}

It is a nice property of the proposition that the bounds it establishes are invariant under basic transformations of proofs:

\begin{remark}\label{rmk:preprocess-proofs}
If $f$ satisfies the conditions of the proposition and $h$ is $\mathbf T$-provably total (e.g.\ primitive recursive) with $h(p)\geq p$ then $h\circ f$ satisfies these conditions as well. Thus proofs of $\exists_y\psi(\overline n,y)$ will even be bigger than $h(f(n))$ for all $n$ above some (possibly increased) bound.\\
This is useful because it allows us to preprocess proofs: Consider a modified notion proof$'$ and a sequence of formulas $\varphi_n$, not necessarily of the form $\varphi(\overline n)$ and not necessarily in the syntactic class $\Sigma_1$. Assume that there is a $\Sigma_1$-formula $\psi(x,y)$ and a primitive recursive function $h$ which transforms any proof$'$ of $\varphi_n$ into (an upper bound for) a proof of $\exists_y\psi(\overline n,y)$. Possibly increasing $h$ we can assume that $h$ is monotone and satisfies $h(p)\geq p$. Using the proposition we may be able to show that $p>h(f(n))$ holds whenever $p$ is a proof of $\exists_y\psi(\overline n,y)$, with $n$ sufficiently large. We want to deduce $q>f(n)$ where $q$ is a proof$'$ of $\varphi_n$. Indeed, $q\leq f(n)$ would imply $h(q)\leq h(f(n))$. This would mean that there exists a proof of $\exists_y\psi(\overline n,y)$ below $h(f(n))$, which we have seen to be false. The proof of Lemma \ref{leq:lower-bound-slow-proofs-ph} contains a detailed application of this argument.
\end{remark}

To conclude this section we illustrate what a simple application of the proposition can yield. Adopting the notation from \cite{solovay81} we have
\begin{equation*}
 \sigma(n,k)=\min\{N\,|\,\ph (\overline k,\overline{n+1},\overline n,\overline N)\text{ is true}\},
\end{equation*}
i.e.\ the number $\sigma(n,k)$ is the smallest witness for the Paris-Harrington Principle with arity $n$ and $k$ colours. We know from \cite[Theorem 3.2]{paris-harrington} that the function $n\mapsto\sigma(n,n)$ eventually dominates any provably total function of Peano Arithmetic. The following result on proof sizes is considerably weaker than (\ref{eq:no-short-proofs-ph}), insofar as it speaks about fixed fragments of Peano Arithmetic.

\begin{corollary}\label{cor:minimal-proofs-fragments-immediate}
For any number $k$ the (total) function
\begin{equation*}
n\mapsto\parbox[t]{10cm}{\raggedright ``the smallest Gödel number of a proof of the\\ \raggedleft $\Sigma_1$-formula $\exists_N\ph (\overline n,\overline n+1,\overline n,N)$ by $\Sigma_k$-induction"}
\end{equation*}
eventually dominates any provably total function of Peano Arithmetic.
\end{corollary}
\begin{proof}
Let $f$ be an arbitrary $\pa$-provably total function. Assume that $f(n)\geq n$ holds for all $n$, possibly after replacing $f$ by the function $n\mapsto\max\{f(n),n\}$. We apply Proposition \ref{prop:lower-bounds-proof-general} to the usual proof predicate $\pro_{\isigma_k}(p,\varphi)$ for the theory of $\Sigma_k$-induction (or rather to a $\Pi_1$-formula that is equivalent to $\pro_{\isigma_k}(p,\varphi)$ over $\isigma_1$), to the theory $\mathbf T=\pa$, to the formula $\psi(x,y)\equiv \ph (x,x+1,x,y)$, and to the function $f$. Then $n\mapsto\sigma(n,n)$ is the function $F_\psi$ of Proposition \ref{prop:lower-bounds-proof-general}. The assumptions of the proposition are satisfied: It is well known that Peano Arithmetic proves uniform $\Sigma_1$-reflection for the theory $\isigma_k$ (see e.g.\ \cite[Corollary I.4.34]{hajek91}). For any $\pa$-provably total function $g$ the composition $g\circ f$ is $\pa$-provably total as well, and thus indeed dominated by $n\mapsto\sigma(n,n)$. The result of Proposition \ref{prop:lower-bounds-proof-general} is nothing but the claim of the corollary.
\end{proof}

The bound of the corollary is reasonably accurate, in the sense that the function computing the minimal proofs is not much faster than the provably total functions of Peano Arithmetic: Recall that $\ph (k,m,n,N)$ is $\Delta_1$ in $\isigma_1$. Thus not only $\sigma(n,n)$ itself but the witnesses to all unbounded quantifiers of the $\Sigma_1$-formula $\exists_N\ph (\overline n,\overline n+1,\overline n,N)$ are bounded by a primitive recursive function in $n$ and $\sigma(n,n)$. Furthermore, the $\Sigma_1$-completeness theorem is established by a primitive recursive construction of proofs. Thus there is a primitive recursive function $h:\mathbb N^2\rightarrow\mathbb N$ such that $h(n,\sigma(n,n))$ is the Gödel number of a proof of $\exists_N\ph (\overline n,\overline n+1,\overline n,N)$ in the theory $\isigma_0$.

\section{No Short Proofs for Instances of the Paris-Harrington Principle}\label{sect:no-short-proofs-ph}

In this section we refine Corollary \ref{cor:minimal-proofs-fragments-immediate} by varying $k$ alongside with $n$. On first sight it may seem astonishing that Proposition \ref{prop:lower-bounds-proof-general}, which only deals with one proof predicate at a time, can be used to this effect. We will see, however, that a single proof predicate can inform us about proofs in various theories: The slow $\pa$-proofs that we will introduce penalize the use of complex induction axioms by a drastic increase in proof length, thus creating an interplay between proof length and the amount of induction used in the proof.\\
Before we can define the notion of a slow proof we need some preliminaries on ordinal notations and the fast-growing hierarchy of functions. Ordinal notations are required for the ordinals below $\varepsilon_0$, the smallest fixed point of the function $\alpha\mapsto\omega^\alpha$. As usual they will be based on the Cantor normal form
\begin{equation*}
\alpha=\omega^{\alpha_1}\cdot n_1+\dots +\omega^{\alpha_k}\cdot n_k\qquad\text{with $k\in\mathbb N$, $n_i\in\mathbb N\backslash\{0\}$ and $\alpha_1>\dots >\alpha_k$.}
\end{equation*}
Crucially, $\alpha<\varepsilon_0$ implies $\alpha_1<\alpha$ so that the Cantor normal form inductively yields finite term notations. Basic ordinal arithmetic can be translated into syntactic operations on these terms. The operations are sufficiently elementary to make ordinal arithmetic available in the theory $\isigma_1$, after arithmetization of the finite term syntax. In fact, Sommer in \cite[Sections 2 and 3]{sommer95} shows that theories much weaker than $\isigma_1$ suffice if one encodes the terms efficiently. In this paper we are not interested in very weak theories, but it is nevertheless convenient to adopt the encoding of Sommer: This allows us to use his $\Delta_0$-definition of the functions in the fast-growing hierarchy.\\
We remark that the ordinal arithmetic of \cite{sommer95} includes fundamental sequences: The fundamental sequence $(\{\alpha\}(n))_{n\in\mathbb N}$ of a limit ordinal $\alpha$ is a strictly increasing sequence of ordinals with supremum $\alpha$. Precisely, any limit ordinal $\alpha$ can uniquely be written as $\alpha=\beta+\omega^\gamma\cdot(k+1)$ where $\gamma>0$ is the smallest exponent of the Cantor normal form of $\alpha$, and $\beta$ contains the larger summands. We then have
\begin{alignat*}{3}
 \{\beta+\omega^\gamma\cdot(k+1)\}(n) & = \beta+\omega^\gamma\cdot k+\omega^\delta\cdot(n+1)\quad && \text{if $\gamma=\delta+1$},\\
 \{\beta+\omega^\gamma\cdot(k+1)\}(n) & = \beta+\omega^\gamma\cdot k+\omega^{\{\gamma\}(n)}\quad && \text{if $\gamma$ is a limit}.
\end{alignat*}
For zero and successor ordinals one sets $\{0\}(n):=0$ and $\{\beta+1\}(0):=\beta$.\\
Next, consider the ``stack of $\omega$'s"-function defined by the recursion
\begin{equation*}
\omega_0^\alpha=\alpha,\qquad \omega_{n+1}^\alpha=\omega^{\omega_n^\alpha}.
\end{equation*}
As usual, $\omega_n$ abbreviates $\omega_n^1$. This function is not part of the ordinal arithmetic encoded by Sommer (although it is, of course, part of his meta-theory). Since Sommer does encode the function $\alpha\mapsto\omega^\alpha$ it is immediate to make the function $(n,\alpha)\mapsto\omega_n^\alpha$ (operating on the codes) available in $\isigma_1$. However, we will need more, namely a $\Delta_0$-formula defining the graph and explicit bounds on the values of this function. Write $\ulcorner\alpha\urcorner$ for the term notation of $\alpha$, represented as a list with digits from $\{1,\dots ,4\}$ as in \cite{sommer95}. Then $\omega_n^\alpha$ is represented by the following concatenation of lists:
\begin{equation*}
\ulcorner\omega_n^\alpha\urcorner=\langle\underbrace{4,\dots ,4}_{\mathclap{\text{$n$ characters $4$}}}\rangle^\frown{\ulcorner\alpha\urcorner}\vphantom{\rangle}^\frown\langle\underbrace{3,1,\dots ,3,1}_{\mathclap{\text{$n$ alternations}}}\rangle
\end{equation*}
Indeed, with each character $4$ we move to the exponent of the leftmost summand of the Cantor normal form, while $3$ instructs us to leave the exponent and look at the corresponding coefficient, which in the present case is always $1$ (represented by the base two notation of $1$, which happens to be the list $\langle 1\rangle$ itself). Now to verify the relation $\omega_n^\alpha=\beta$ we only have to compare digits in the sequence representations of $\alpha$ and $\beta$, and this can be cast into a $\Delta_0$-formula (see \cite[Section 2.2]{sommer95}). Using \cite[Proposition 2.1]{sommer95}, which relates the code of a list of digits to its length, we can also establish the following inequality between the codes of $\alpha$ and $\omega_n^\alpha$:
\begin{equation}\label{eq:bound-codes-stack-of-omega}
\isigma_1\vdash\forall_{n,\alpha}\,\omega_n^\alpha\leq 4^{3n+1}\cdot(\alpha+1).
\end{equation}
Let us remark that we do not extend the ordinal notation system by a symbol for $\varepsilon_0$, in order to keep it closed under the usual operations of ordinal arithmetic. By a harmless abuse of notation we will sometimes refer to the ``fundamental sequence" of $\varepsilon_0$, which we define as $\{\varepsilon_0\}(n):=\omega_{n+1}$.\\
Using fundamental sequences we can define the fast-growing hierarchy of functions indexed by ordinals below and including $\varepsilon_0$. The definition varies slightly within the literature; our version differs from the classic \cite{wainer70,schwichtenberg71} and coincides e.g.\ with \cite{sommer95}:
\begin{gather*}
F_0(x) := x+1,\\
F_{\alpha+1}(x) := F_\alpha^{x+1}(x),\\
F_\lambda(x) := F_{\{\lambda\}(x)}(x)\quad\text{for $\lambda$ a limit ordinal}.
\end{gather*}
Here and in the following an exponent to a function symbol denotes the number of times the function is to be iterated. Given an arithmetization of ordinal arithmetic it is easy to define the graph of $(\alpha,x,i)\mapsto F_\alpha^i(x)$ by a $\Sigma_1$-formula in the language of first-order arithmetic: As described in \cite[Section 4.1]{sommer95-thesis} one can compute $F_\alpha^i(x)$ by simplifying expressions of the form $F_{\alpha_1}^{i_1}(F_{\alpha_2}^{i_2}(\cdots(F_{\alpha_k}^{i_k}(z))\cdots))$, so one only needs to state the existence of such a computation sequence. What is remarkable is that the size of an (improved) computation sequence can be bounded by a polynomial in the value of $F_\alpha^i(x)$. This is worked out in \cite[Appendix A]{sommer95-thesis} (see also the less detailed \cite[Section 5.2]{sommer95}) and leads to a $\Delta_0$-formula $F_\alpha^i(x)=y$ with free variables $x,y,\alpha,i$ which defines the functions $F_\alpha$ for $\alpha <\varepsilon_0$, as well as their iterations. By \cite[Theorem 5.3]{sommer95} the defining equations of the fast-growing hierarchy are provable in $\isigma_1$ (under the assumption that the involved computations terminate, which is of course unprovable in $\isigma_1$). As Sommer only encodes the hierarchy below $\varepsilon_0$ we should show separately that the formula
\begin{equation*}
 \feps(x)=y\quad:\equiv\quad\exists_\alpha(\alpha=\omega_{x+1}\land F_\alpha(x)=y)
\end{equation*}
is $\Delta_0$ in $\isigma_1$: The only task is to bound the existentially quantified $\alpha$. By \cite[Lemma 2.3, Proposition 2.12]{rathjen13} the inequalities
\begin{equation*}
 F_{\omega_{x+1}}(x)\geq F_\omega(x)\geq F_2(x)=2^{x+1}\cdot(x+1)-1\geq 2^{x+1}\qquad\text{for $x\geq 1$}
\end{equation*}
are provable in $\isigma_1$. Combining this with (\ref{eq:bound-codes-stack-of-omega}) we obtain
\begin{equation}\label{eq:feps-delta-zero}
 \isigma_1\vdash x\geq 1\,\rightarrow\,(\feps(x)=y\leftrightarrow\exists_{\alpha\leq y^6\cdot 4\cdot(\ulcorner 1\urcorner+1)}(\alpha=\omega_{x+1}\land F_\alpha(x)=y)),
\end{equation}
where $\ulcorner 1\urcorner$ denotes the code of the ordinal $1$.\\
Writing $\langle\cdot,\cdot\rangle$ for the Cantor pairing function with projections $\pi_1(\cdot),\pi_1(\cdot)$ we can now define slow proofs in Peano Arithmetic. The idea is to penalize the use of complex induction axioms by a drastic increase in proof length, and thus to create an interplay between proof size and the amount of induction used in the proof.

\begin{definition}[cf.\ {\cite{rathjen13}}]
 A pair $\langle q,N\rangle$ is a slow $\pa$-proof of a formula $\varphi$ if there is a number $n$ such that we have $N=\feps(n)$ and such that $q$ codes a (usual) proof of $\varphi$ in the theory $\isigma_{n+1}$. This notion is defined by the formula
\begin{equation*}
 \pro_\pa^\diamond(p,\varphi):\equiv\exists_x(\pro_{\isigma_{x+1}}(\pi_1(p),\varphi)\land\feps(x)=\pi_2(p)),
\end{equation*}
which is $\Delta_1$ in $\isigma_1$ since by \cite[Proposition 5.4]{sommer95} the second conjunct implies the bound $x\leq\pi_2(p)$.
\end{definition}

For a formula $F(x)=y$ let us abbreviate $\exists_y F(x)=y$ by $F(x)\!\downarrow$. Also, we write $\pr_{\isigma_x}(\varphi)$ for the formula $\exists_p\pro_{\isigma_x}(p,\varphi)$. It is easy to see that the slow provability predicate
\begin{equation*}
 \pr_\pa^\diamond(\varphi):\equiv\exists_p\pro_\pa^\diamond(p,\varphi)
\end{equation*}
satisfies the equivalence
\begin{equation*}
 \isigma_1\vdash\pr_\pa^\diamond(\varphi)\leftrightarrow\exists_x(\pr_{\isigma_{x+1}}(\varphi)\land\feps(x)\!\downarrow).
\end{equation*}
The slow uniform $\Sigma_1$-reflection principle
\begin{equation*}
 \drfn:\equiv\forall_\varphi(\text{``$\varphi$ is a closed $\Sigma_1$-formula''}\land\pr_\pa^\diamond(\varphi)\rightarrow\true_{\Sigma_1}(\varphi))
\end{equation*}
and the slow consistency statement
\begin{equation*}
 \con^\diamond(\pa):\equiv\neg\pr_\pa^\diamond(\overline{\ulcorner 0=1\urcorner})
\end{equation*}
can be characterized as
\begin{equation}\label{eq:slow-rfn-rfn-fragments}
 \isigma_1\vdash\drfn\leftrightarrow\forall_x(\feps(x)\!\downarrow\,\rightarrow\rfn(\isigma_{x+1}))
\end{equation}
and
\begin{equation*}
 \isigma_1\vdash\con^\diamond(\pa)\leftrightarrow\forall_x(\feps(x)\!\downarrow\,\rightarrow\con(\isigma_{x+1})).
\end{equation*}
As the last equivalence reveals the notion of slow $\pa$-proof comes from the article \cite{rathjen13} by S.-D.\ Friedman, Rathjen and Weiermann: These authors introduce the slow consistency statement
\begin{equation*}
 \con^*(\pa)\equiv\forall_x(\feps(x)\!\downarrow\,\rightarrow\con(\isigma_x))
\end{equation*}
and show that we have
\begin{equation}\label{eq:slow-con-weaker-con}
 \pa+\con^*(\pa)\nvdash\con(\pa).
\end{equation}
It has been pointed out by Michael Rathjen \cite{rathjen-miscellanea-slow-consistency} that slow provability satisfies the G\"odel-L\"ob conditions, provably so in $\isigma_1$. In many respects it thus behaves as the usual provability predicate for Peano Arithmetic. The index shift between our $\con^\diamond(\pa)$ and the formula $\con^*(\pa)$ of \cite{rathjen13} has been introduced to improve the bounds on proof sizes that we are about to establish.\\
The central ingredient to our bounds on proof sizes is a computational analysis of slow reflection. Since this analysis is independent and somewhat technical we defer it to Section \ref{sect:provably-total-slow-reflection} below. In the present section we will only use the following result of this analysis:

\newtheorem*{thm:total-functions-slow-reflection-mod}{Theorem \ref{thm:total-functions-slow-reflection-mod}}
\begin{thm:total-functions-slow-reflection-mod}
 For any provably total function $g$ of $\pa+\drfn$ there is a number $N$ such that we have
\begin{equation*}
 g(\feps(n\dotminus 1))\leq \feps(n)\qquad\text{for all $n\geq N$}.
\end{equation*}
In particular any provably total function of the theory $\pa+\drfn$ is eventually dominated by $\feps$.
\end{thm:total-functions-slow-reflection-mod}

The reader who prefers to see all proofs in order may go through Section \ref{sect:provably-total-slow-reflection} now and return to this point afterwards. In the rest of this section we show how results about proof sizes in fragments of Peano Arithmetic can be deduced. It is worth observing that a weaker version of Theorem \ref{thm:total-functions-slow-reflection-mod} suffices for these applications: Namely, it would be enough to bound the provably total functions of $\isigma_1+\drfn$ rather than those of $\pa+\drfn$. However, as a result in its own right Theorem \ref{thm:total-functions-slow-reflection-mod} is certainly more satisfying with the stronger base theory. Let us now investigate the size of proofs of the formulas $\feps(\overline n)\!\downarrow$. Afterwards we will come to the slightly more subtle case of the Paris-Harrington Principle:

\begin{lemma}\label{leq:lower-bound-slow-proofs}
 There is a number $N$ such that we have
\begin{equation*}
p>\langle\feps(n\dotminus 1),\feps(n\dotminus 1)\rangle\quad\text{for any slow $\pa$-proof $p$ of $\feps(\overline n)\!\downarrow$ with $n\geq N$}.
\end{equation*}
\end{lemma}
To avoid misunderstanding we recall that $\langle\cdot,\cdot\rangle$ denotes the Cantor pairing.
\begin{proof}
 We apply Proposition \ref{prop:lower-bounds-proof-general} to the proof predicate $\pro_\pa^\diamond(p,\varphi)$, the theory $\mathbf T=\isigma_1+\drfn$, the formula $\psi(x,y)\equiv \feps(x)=y$ (so that $F_\psi$ is the function $\feps$), and the function $n\mapsto\langle\feps(n\dotminus 1),\feps(n\dotminus 1)\rangle$ at the place of $f$. Let us verify the assumptions of Proposition \ref{prop:lower-bounds-proof-general}: By (\ref{eq:slow-rfn-rfn-fragments}) we have
\begin{equation*}
 \isigma_1+\rfn(\pa)\vdash\drfn,
\end{equation*}
where $\rfn(\pa)$ denotes the usual uniform $\Sigma_1$-reflection principle for Peano Arithmetic. This shows that the theory $\isigma_1+\drfn$ is sound. Next, using \cite[Proposition 2.5]{solovay81} we have
 \begin{equation*}
 n\leq\feps(n\dotminus 1)\leq\langle\feps(n\dotminus 1),\feps(n\dotminus 1)\rangle.
 \end{equation*}
Finally, consider an arbitrary function $g$ that is provably total in the theory $\isigma_1+\drfn$. We have to show that there is a number $N$ such that we have
\begin{equation*}
g(\langle\feps(n\dotminus 1),\feps(n\dotminus 1)\rangle)\leq\feps(n)\qquad\text{for all $n\geq N$}.
\end{equation*}
This follows from Theorem \ref{thm:total-functions-slow-reflection-mod}, applied not to $g$ itself but rather to the function $m\mapsto g(\langle m,m\rangle)$, which is still provably total in the theory $\isigma_1+\drfn$. Now Proposition \ref{prop:lower-bounds-proof-general} gives us precisely the claim.
\end{proof}

It is easy to deduce bounds for proofs in the fragments of Peano Arithmetic:

\begin{theorem}\label{thm:no-small-proof-feps-defined}
 There is a number $N$ such that for all $n\geq N$ no proof of the statement $\feps(\overline n)\!\downarrow$ in the theory $\isigma_n$ can have code less than or equal to $\feps(n\dotminus 1)$.
\end{theorem}
\begin{proof}
 We can assume that the bound $N$ in Lemma \ref{leq:lower-bound-slow-proofs} is bigger than zero. Let us show that the present result holds with the same bound: Aiming at a contradiction, suppose that $q\leq\feps(n\dotminus 1)$ is an $\isigma_n$-proof of the formula $\feps(\overline n)\!\downarrow$, for some $n\geq N$. By definition $\langle q,\feps(n\dotminus 1)\rangle$ is a slow $\pa$-proof of $\feps(\overline n)\!\downarrow$. Thus the inequality
\begin{equation*}
 \langle q,\feps(n\dotminus 1)\rangle\leq\langle\feps(n\dotminus 1),\feps(n\dotminus 1)\rangle
\end{equation*}
contradicts Lemma \ref{leq:lower-bound-slow-proofs}.
\end{proof}

To deduce corresponding results for instances of the Paris-Harrington Principle, recall the function $(n,k)\mapsto\sigma(n,k)$ defined just before Corollary \ref{cor:minimal-proofs-fragments-immediate} above. We need to link this function to the function $\feps$:

\begin{lemma}[{\cite{solovay81}}]\label{lem:bound-solovay-sigmas-new}
 We have
 \begin{equation*}
 \feps(n)\leq\sigma(n+2,10^{35n^2})\leq\sigma(n+3,8)\qquad\text{for all $n\geq 15$.}
\end{equation*}
\end{lemma}
\begin{proof}
This is the result of \cite[Lemma 3.6, Theorem 3.10]{solovay81}, except that \cite{solovay81} works with a slightly different version of fundamental sequences, setting
\begin{equation*}
 \{\beta+\omega^\gamma\cdot(k+1)\}(n) = \beta+\omega^\gamma\cdot k+\omega^\delta\cdot n\quad\text{in case $\gamma=\delta+1$}.
\end{equation*}
With this definition, descending to the $n$-th member of the fundamental sequence can introduce a coefficient (bounded by) $n$. In our case the new coefficients are bounded by $n+1$. The overall bound $\sigma(n+2,10^{23n^2})$ of \cite[Lemma 3.6]{solovay81} then increases to our $\sigma(n+2,10^{35n^2})$.\\
Let us describe the concrete changes that are necessary (the reader will have to consult \cite{solovay81} for context): First, the bound of \cite[Proposition 2.9]{solovay81} increases from $|T_{k,c,n}|\leq (n+1)_k^c$ to $|T_{k,c,n}|\leq (n+2)_k^c$. At the same time the rather generous bound $|T_{k,c,n}|\leq 2_{k-1}^{(n^{6c})}$ of \cite[Proposition 2.10]{solovay81} remains valid without change. Thus \cite[Lemma 3.1]{solovay81} remains valid, and so does \cite[Lemma 3.2.1]{solovay81}. A small change is required in \cite[Lemma 3.2.2]{solovay81}: We need to weaken the condition $g(x_0,\dots ,x_{n-1})\leq x_0$ to $g(x_0,\dots ,x_{n-1})\leq x_0+1$. It is easy to see that $g$ is then controlled by an $(n+1,10^5)$-algebra (instead of an $(n+1,10^4)$-algebra). Consequently, \cite[Lemma 3.2.3]{solovay81} now constructs an $(n+1,10^{5c})$-algebra. One can check that \cite[Lemma 3.4]{solovay81} remains valid in spite of the prior changes: The bound of \cite[Lemma 3.2.3]{solovay81} is still strong enough for the base case of the proof; in the step, the bound is generous enough to accomodate the fact that $G_3$ is now an $(n+2,10^5)$-algebra. It follows that \cite[Theorem 3.5]{solovay81} remains unchanged: For $n,k\geq 1$ the function $F_{\omega_n^k}$ is captured by an $(n+2,10^{n\cdot(12n+2k+8)})$-algebra. Parallel to \cite[Lemma 3.6]{solovay81} we can now deduce the desired bound: We have $\{\omega_{n+1}\}(n)=\omega_n^{n+1}$ and thus $\feps(n)=F_{\omega_{n+1}}(n)=F_{\omega_n^{n+1}}(n)$ (as opposed to $\feps(n)=F_{\omega_n^n}(n)$ in the original \cite[Lemma 3.6]{solovay81}). Let $G_0$ be an $(n+2,10^{14n^2+20n})$-algebra that captures $F_{\omega_n^{n+1}}$. Let $G_1$ be an $(n+2,7)$-algebra such that $\min(S)\geq 2n+3$ holds whenever $S$ is suitable for $G_1$. In view of
\begin{equation*}
 7\cdot 10^{14n^2+20n}\leq 10^{14n^2+20n+1}\leq 10^{35n^2}\quad\text{(for $n\geq 1$)}
\end{equation*}
we can choose an $(n+2,10^{35n^2})$-algebra $G$ which simulates $G_0$ and $G_1$. If $S$ is suitable for $G$ then we have
\begin{equation*}
 max(S)\geq s_2>s_1\geq F_{\omega_n^{n+1}}(s_0)\geq F_{\omega_n^{n+1}}(n)=\feps(n).
\end{equation*}
This means that the restriction
\begin{equation*}
 G\restriction_{[\feps(n)]^{n+2}}:[\feps(n)]^{n+2}\rightarrow 10^{35n^2}
\end{equation*}
admits no suitable set. Thus we have $\feps(n)<\sigma(n+2,10^{35n^2})$.\\
It remains to check $\sigma(n+2,10^{35n^2})\leq\sigma(n+3,8)$. This is parallel to the proof of \cite[Theorem 3.10]{solovay81}: Observe that we have
\begin{equation*}
 F_3^{n+1}(n+2)\geq F_3(n)\geq 2^{2^n}\geq2^{4\cdot 35n^2}\geq 10^{35n^2}\qquad\text{for $n\geq 15$}.
\end{equation*}
Thus by \cite[Lemma 3.9]{solovay81} each $(n+2,10^{35n^2})$-algebra can be simulated by an $(n+3,8)$-algebra, and this implies the claim. Note that the condition $n\geq 15$ could easily be replaced by a smaller bound.
\end{proof}

This implies the following result, which we will need in our applications:

\begin{corollary}\label{cor:bound-slow-reflection-function-ph}
 For any provably total function $g$ of $\isigma_1+\drfn$ there is a number $N$ such that we have
 \begin{equation*}
 g(\feps(n\dotminus 1))\leq\sigma(n+2,10^{35n^2})\leq\sigma(n+3,8)\qquad\text{for all $n\geq N$}.
\end{equation*}
\end{corollary}
\begin{proof}
 This follows from Theorem \ref{thm:total-functions-slow-reflection-mod} and Lemma \ref{lem:bound-solovay-sigmas-new}.
\end{proof}

Similar to Lemma \ref{leq:lower-bound-slow-proofs}, slow proofs of certain instances of the Paris-Harrington Principle must be long:

\begin{lemma}\label{leq:lower-bound-slow-proofs-ph}
The following holds:
\begin{enumerate}[label=(\alph*)]
 \item There is a number $K'$ such that we have $p>\langle\feps(n\dotminus 1),\feps(n\dotminus 1)\rangle$ for any slow $\pa$-proof $p$ of $\exists_N\ph (\overline{10^{35n^2}},\overline{n+3},\overline{n+2},N)$ with $n\geq K'$.
\item There is a number $K'$ such that we have $p>\langle\feps(n\dotminus 1),\feps(n\dotminus 1)\rangle$ for any slow $\pa$-proof $p$ of $\exists_N\ph (8,\overline{n+4},\overline{n+3},N)$ with $n\geq K'$.
\end{enumerate}
\end{lemma}
\begin{proof}
 We only show (a). The proof of (b) is similar and somewhat easier. Compared to the proof of Lemma \ref{leq:lower-bound-slow-proofs}, the main subtlety is that the formulas
\begin{equation*}
 \varphi_n:\equiv\exists_N\ph (\overline{10^{35n^2}},\overline{n+3},\overline{n+2},N)
\end{equation*}
are not of the form $\varphi(\overline n)$, i.e.\ parametrized by the $n$-th numeral. To make Proposition \ref{prop:lower-bounds-proof-general} applicable we need to preprocess proofs of these formulas, as sketched in Remark \ref{rmk:preprocess-proofs}: Let $e(x)=z$ be a $\Sigma_1$-formula such that we have
\begin{equation*}
 \mathbb N\vDash e(\overline n)=\overline k\qquad\Leftrightarrow\qquad k=10^{35n^2}
\end{equation*}
and $\isigma_1\vdash\forall_x\exists_z\, e(x)=z$. In view of the latter, the witnesses to all unbounded quantifiers of the $\Sigma_1$-formula $\exists_z\,e(\overline n)=z$ are bounded by a primitive recursive function in $n$. By the proof of $\Sigma_1$-completeness there is a primitive recursive function $p_e:\mathbb N^2\rightarrow\mathbb N$ such that $p_e(n,k)$ is an $\isigma_k$-proof of $e(\overline n)=\overline{10^{35n^2}}$.\\
Next, let $\psi(x,y)$ be a $\Sigma_1$-formula with
\begin{equation}\label{eq:sigma1-equivalence-paris-harrington}
 \isigma_1\vdash \psi(x,y)\leftrightarrow\exists_z(e(x)=z\land\ph (z,x+3,x+2,y)).
\end{equation}
Following Remark \ref{rmk:preprocess-proofs}, we need a primitive recursive function $h:\mathbb N\rightarrow\mathbb N$ which transforms a slow $\pa$-proof of $\varphi_n$ into a slow $\pa$-proof of $\exists_y\psi(\overline n,y)$. Let us first construct a primitive recursive function $h':\mathbb N^2\rightarrow\mathbb N$ such that $h'(k,q)$ is an $\isigma_{k+1}$-proof of $\exists_y\psi(\overline n,y)$ if $q$ is an $\isigma_{k+1}$-proof of $\varphi_n$: Given a proof $q$ as described, we can read off its end formula $\varphi_n$ and then the number $n$. Recall that $p_e(n,k+1)$ is an $\isigma_{k+1}$-proof of $e(\overline n)=\overline{10^{35n^2}}$. Combining this with $q$ and introducing an existential quantifier yields an $\isigma_{k+1}$-proof of
\begin{equation*}
 \exists_z(e(\overline n)=z\land\exists_N\ph (z,\overline{n+3},\overline{n+2},N)).
\end{equation*}
It is not unreasonable to assume that $\overline{n+3}$ (resp.\ $\overline{n+2}$) is the same term as $\overline n+3$ (resp.\ $\overline n+2$). Even if not, there are primitive recursive functions which map a pair $(k,n)$ to $\isigma_{k+1}$-proofs of $\overline{n+3}=\overline n+3$ and $\overline{n+2}=\overline n+2$. We then apply the equality axioms and prefix the existentially quantified $N$, giving
an $\isigma_{k+1}$-proof of
\begin{equation*}
 \exists_y\exists_z(e(\overline n)=z\land\ph (z,\overline n+3,\overline n+2,y)).
\end{equation*}
Invoking the equivalence (\ref{eq:sigma1-equivalence-paris-harrington}) we get the desired proof $h'(k,q)$ of $\exists_y\psi(\overline n,y)$. Now to construct $h$, assume that $p=\langle q,M\rangle$ is a slow $\pa$-proof of $\varphi_n$. By definition there is an $m\leq M$ such that $q$ is an $\isigma_{m+1}$ proof of $\varphi_n$ and such that we have $\feps(m)=M$. Recall that the relation $\feps(x)=y$ is primitive recursively decidable, and that $\feps$ is strictly monotone. Thus we can primitive recursively determine the unique $m$ with the stated property. Now it suffices to set
\begin{equation*}
  h(p):=\langle h'(m,q), M\rangle.                                                                                                                                 \end{equation*}
We need to increase $h$ to make it monotone and ensure $h(p)\geq p$. Clearly, the increased function still satisfies the following: If $p$ is a slow $\pa$-proof of $\varphi_n$ then there is a slow $\pa$-proof of $\exists_y\psi(\overline n,y)$ below $h(p)$.\\
Now we apply Proposition \ref{prop:lower-bounds-proof-general} to the proof predicate $\pro_\pa^\diamond(p,\varphi)$, the theory $\mathbf T=\isigma_1+\drfn$, the $\Sigma_1$-formula $\psi(x,y)$ defined above, and the function $n\mapsto h(\langle\feps(n\dotminus 1),\feps(n\dotminus 1)\rangle)$ at the place of $f$. In view of (\ref{eq:sigma1-equivalence-paris-harrington}) we have
\begin{equation*}
 \mathbb N\vDash\psi(\overline n,\overline m)\quad\Leftrightarrow\quad\mathbb N\vDash\ph (\overline{10^{35n^2}},\overline{n+3},\overline{n+2},\overline m),
\end{equation*}
so that $F_\psi$ is the function $n\mapsto\sigma(n+2,10^{35n^2})$. Concerning the assumptions of Proposition \ref{prop:lower-bounds-proof-general}, in view of $h(p)\geq p$ (see also the proof of Lemma \ref{leq:lower-bound-slow-proofs}) we have
\begin{equation*}
 h(\langle\feps(n\dotminus 1),\feps(n\dotminus 1)\rangle)\geq n\quad\text{for all $n$}.
\end{equation*}
Coming to the other assumption, let $g$ be any provably total function of $\isigma_1+\drfn$. We must show that $n\mapsto g(h(\langle\feps(n\dotminus 1),\feps(n\dotminus 1)\rangle))$ is eventually dominated by the function $n\mapsto\sigma(n+2,10^{35n^2})$. To see this one applies Corollary \ref{cor:bound-slow-reflection-function-ph} to the function $m\mapsto g(h(\langle m,m\rangle))$, which is still provably total in the theory $\isigma_1+\drfn$. Having verified the assumptions Proposition \ref{prop:lower-bounds-proof-general} gives us a bound $K'$ such that we have
\begin{equation*}
 p'>h(\langle\feps(n\dotminus 1),\feps(n\dotminus 1)\rangle)
\end{equation*}
whenever $p'$ is a slow $\pa$-proof of $\exists_y\psi(\overline n,y)$ with $n\geq K'$. To deduce the claim of (a), let $p$ be a slow $\pa$-proof of $\exists_N\ph (\overline{10^{35n^2}},\overline{n+3},\overline{n+2},N)$, still with $n\geq K'$. As we have seen above, this implies that there is a slow $\pa$-proof of $\exists_y\psi(\overline n,y)$ below $h(p)$. By the bound that we have just established we must have
\begin{equation*}
 h(p)>h(\langle\feps(n\dotminus 1),\feps(n\dotminus 1)\rangle).
\end{equation*}
Since $h$ is monotone this does indeed imply $p>\langle\feps(n\dotminus 1),\feps(n\dotminus 1)\rangle$.
\end{proof}

We can derive the central result of the paper, claim (\ref{eq:no-short-proofs-ph}) from the introduction:

\begin{theorem}\label{thm:no-small-proofs-ph}
The following holds:
\begin{enumerate}[label=(\alph*)]
 \item There is a number $K$ such that for all $n\geq K$ no proof of the formula $\exists_N\ph (\overline{10^{35(n\dotminus 2)^2}},\overline{n+1},\overline n,N)$ in the theory $\isigma_{n\dotminus 2}$ can have G\"odel number less than or equal to $\feps(n\dotminus 3)$.
 \item There is a number $K$ such that for all $n\geq K$ no proof of the formula $\exists_N\ph (\overline{8},\overline{n+1},\overline n,N)$ in the theory $\isigma_{n\dotminus 3}$ can have G\"odel number less than or equal to $\feps(n\dotminus 4)$.
\end{enumerate}
\end{theorem}
\begin{proof}
 We only write out the proof for (a), the proof of (b) being completely parallel: Let $K'$ be the bound from Lemma \ref{leq:lower-bound-slow-proofs-ph}, and set $K:=\max\{K'+2,3\}$. Consider an arbitrary $n\geq K$ and a proof $q$ of $\exists_N\ph (\overline{10^{35(n-2)^2}},\overline{n+1},\overline n,N)$ in the theory $\isigma_{n-2}$. It follows that the pair $\langle q,\feps(n-3)\rangle$ is a slow $\pa$-proof of $\exists_N\ph (\overline{10^{35(n-2)^2}},\overline{n+1},\overline n,N)$. Lemma \ref{leq:lower-bound-slow-proofs-ph} yields
\begin{equation*}
 \langle q,\feps(n-3)\rangle>\langle\feps(n-3),\feps(n-3)\rangle.
\end{equation*}
Since the Cantor pairing is monotone we get $q>\feps(n-3)$, as desired.
\end{proof}

By claim (\ref{eq:short-proofs-ph}) from the introduction both $\exists_N\ph (\overline{10^{35(n\dotminus 2)^2}},\overline{n+1},\overline n,N)$ and $\exists_N\ph (\overline{8},\overline{n+1},\overline n,N)$ have short proofs in $\isigma_{n\dotminus 1}$. The fragment $\isigma_{n\dotminus 2}$ in part (a) of the theorem is thus optimal. Concerning (b), it is currently open whether $\exists_N\ph (\overline{8},\overline{n+1},\overline n,N)$ has a short proof in $\isigma_{n\dotminus 2}$. In any case the parameters of the Paris-Harrington Principle leave room for variation: For example, the bounds established by Loebl and Ne\v{s}et\v{r}il \cite{nesetril92} (with shorter proofs than in \cite{solovay81}) lead to similar results.

\section{The Provably Total Functions of Slow Reflection}\label{sect:provably-total-slow-reflection}

The goal of this section is to provide a proof of Theorem \ref{thm:total-functions-slow-reflection-mod}, which we already used (but did not prove) in the previous section. We will need the following characterization of uniform $\Sigma_1$-reflection over the fragments of Peano Arithmetic:

\begin{proposition}\label{prop:reflection-fragments-fast-growing}
 We have
\begin{equation*}
 \isigma_1\vdash\forall_x(F_{\omega_x}\!\downarrow\,\leftrightarrow\,\rfn(\isigma_x)).
\end{equation*}
\end{proposition}
\begin{proof}
It is known that the equivalence $F_{\omega_n}\!\downarrow\,\leftrightarrow\,\rfn(\isigma_n)$ for fixed $n$ is provable in $\isigma_1$ (and in weaker theories): A model-theoretic proof can be found in \cite{paris80} or \cite[Proposition 6.8]{sommer95}. For a proof-theoretic approach (via iterated reflection principles) we refer to \cite[Theorem 1, Proposition 7.3, Remark 7.4]{beklemishev03}. The author has found no fully explicit argument that the formalization is uniform in $n$. We provide a detailed proof of this fact in \cite{freund-characterize-reflection}: This is a proof-theoretic argument, formalizing the infinitary proof system from \cite{buchholz-wainer-87} by the method of \cite{buchholz91}.
\end{proof}

Using this result and (\ref{eq:slow-rfn-rfn-fragments}) we can view slow reflection as a statement about the fast-growing hierarchy of functions:

\begin{corollary}\label{cor:slow-reflection-connect-fast-growing}
 We have
\begin{equation*}
 \isigma_1\vdash\drfn\,\leftrightarrow\,\forall_x(\feps(x)\!\downarrow\,\rightarrow F_{\omega_{x+1}}\!\downarrow).
\end{equation*}
\end{corollary}

Note that the ``index shift'', stemming from the definition of slow proof, is indeed optimal: In view of $\feps(x)\simeq F_{\omega_{x+1}}(x)$ we can deduce
\begin{equation*}
 \isigma_1\vdash\forall_x(\feps(x)\!\downarrow\,\rightarrow F_{\omega_{x+2}}\!\downarrow)\,\rightarrow\,\forall_y\feps(y)\!\downarrow
\end{equation*}
by induction on $y$. Thus a stronger slow reflection statement would collapse into the usual notion of $\Sigma_1$-reflection over Peano Arithmetic. This explains why our bounds on proof size are relatively sharp.\\
Our next goal is to transform the $\Pi_2$-statement $\forall_x(\feps(x)\!\downarrow\,\rightarrow F_{\omega_{x+1}}\!\downarrow)$ into a formula which defines a unary function.

\begin{definition}
The inverse $\feps^{-1}$ of the function $\feps$ (see \cite[Definition 3.2]{rathjen13}) is given by
\begin{equation*}
\feps^{-1}(x):=\max(\{z\leq x\, |\, \exists_{w\leq x}\feps(z)=w\}\cup\{0\}).
\end{equation*}
Note that the $\Delta_0$-definition of $\feps$ yields a $\Delta_0$-definition of $\feps^{-1}$. To define a slow variant $\feps^\diamond$ of the function $\feps$ we set
\begin{equation*}
\feps^\diamond(x):=F_{\omega_{\feps^{-1}(x)+1}}(x),
\end{equation*}
which has the $\Sigma_1$-definition
\begin{equation*}
\feps^\diamond(x)=y\quad\Leftrightarrow\quad\exists_z(z=\feps^{-1}(x)\land\exists_\alpha(\alpha=\omega_{z+1}\land F_\alpha(x)=y)).
\end{equation*}
Clearly, $z$ is bounded by $x$. In view of (\ref{eq:feps-delta-zero}) the code of $\alpha$ is bounded by a polynomial in $x$. Thus the given definition of $\feps^\diamond$ is $\Delta_0$ in $\isigma_1$.
\end{definition}

We remark that the idea behind $\feps^\diamond$ is similar to Simmons' slow variant of the Ackermann function in \cite[Paragraph 2]{simmons10}. Let us now connect $\feps^\diamond$ with the slow reflection principle:

\begin{proposition}\label{prop:slow-reflection-unary-function}
 We have
\begin{equation*}
 \isigma_1\vdash\drfn\leftrightarrow\feps^\diamond\!\downarrow.
\end{equation*}
\end{proposition}
\begin{proof}
By Corollary \ref{cor:slow-reflection-connect-fast-growing} the claim of the proposition is equivalent to
\begin{equation*}
 \isigma_1\vdash\forall_x(\feps(x)\!\downarrow\,\rightarrow F_{\omega_{x+1}}\!\downarrow)\leftrightarrow\feps^\diamond\!\downarrow.
\end{equation*}
To show the direction ``$\rightarrow$'' we work in $\isigma_1$ and assume that the formula $\forall_x(\feps(x)\!\downarrow\,\rightarrow F_{\omega_{x+1}}\!\downarrow)$ holds. We have to prove $\feps^\diamond(x)\!\downarrow$ for an arbitrary $x$. The finitely many $x<\feps(0)$ are treated by $\Sigma_1$-completeness. For $x\geq\feps(0)$ the set $\{z\leq x\,|\,\exists_{w\leq x}\feps(z)=w\}$ is non-empty, so $\feps^{-1}(x)=:z$ is an element of this set. In particular it follows that $\feps(z)$ is defined. Then the assumption $\forall_x(\feps(x)\!\downarrow\,\rightarrow F_{\omega_{x+1}}\!\downarrow)$ tells us that $F_{\omega_{z+1}}$ is total. Thus $F_{\omega_{z+1}}(x)$ is defined, as required for $\feps^\diamond(x)\!\downarrow$.\\
For the direction ``$\leftarrow$'', assume that the function $\feps^\diamond$ is total, let $x$ be arbitrary, and assume that $\feps(x)$ is defined. We have to prove that $F_{\omega_{x+1}}$ is total. By \cite[Lemma 2.3]{rathjen13} it suffices to show that $F_{\omega_{x+1}}(y)$ is defined for arbitrarily large $y$. Since $\feps(x)$ was assumed to be defined, we may consider an arbitrary $y$ above this value. Then we have $x\leq\feps^{-1}(y)=:z$. Invoking the totality of $\feps^\diamond$ we learn that $\feps^\diamond(y)=F_{\omega_{z+1}}(y)$ is defined. It follows by \cite[Lemma 2.4, Proposition 2.12, Lemma 2.3]{rathjen13} that $F_{\omega_{x+1}}(y)$ is defined (and has value at most $\feps^\diamond(y)$).
\end{proof}

By the parenthesis at the end of the proof, the function $\feps^\diamond$ dominates $F_{\omega_{x+1}}$ for values above $\feps(x)$. In other words, $\feps^\diamond$ eventually dominates any provably total function of Peano Arithmetic. In particular we have
\begin{equation*}
 \pa\nvdash\drfn.
\end{equation*}
Since slow reflection implies slow consistency this was already known by \cite[Proposition 3.3]{rathjen13}. It is important that the argument we just gave does not formalize in Peano Arithmetic: To show that $\feps^\diamond$ dominates $F_{\omega_{x+1}}$ we had to know that $\feps(x)$ is defined. If this was different then $\feps^\diamond\!\downarrow$ would imply $\feps\!\downarrow$, contradicting the result that we are about to prove.\\
Recall that our goal is to bound the provably total functions of the theory $\pa+\drfn$, or equivalently those of $\pa+\feps^\diamond\!\downarrow$. It is a classical result that any provably total function of Peano Arithmetic is dominated by some function $F_\alpha$ with $\alpha<\varepsilon_0$ from the fast-growing hierarchy. To analyse $\pa+\drfn$ we build an analogous hierarchy on top of $\feps^\diamond$:

\begin{definition}
 By induction on $\alpha<\varepsilon_0$ we define functions $F_{\varepsilon_0+\alpha}^\diamond$: Set
\begin{align*}
 F_{\varepsilon_0+0}^\diamond(n) & := F_{\varepsilon_0}^\diamond(n),\\
 F_{\varepsilon_0+\alpha+1}^\diamond(n) & := (F_{\varepsilon_0+\alpha}^\diamond)^{n+1}(n),\\
 F_{\varepsilon_0+\alpha}^\diamond(n) & := F_{\varepsilon_0+\{\alpha\}(n)}^\diamond(n)\quad\text{for $\alpha$ limit},
\end{align*}
where the superscript $n+1$ denotes the number of iterations and $\{\alpha\}(n)$ refers to the fundamental sequence of $\alpha$, as defined at the beginning of Section 2.
\end{definition}

To make use of this hierarchy we will need some monotonicity properties. These will involve the ``step down"-relation from \cite[Section 2]{solovay81} (with slightly different fundamental sequences) or \cite[Section 2]{rathjen13}: We write $\beta\rightarrow_n\gamma$ to express that there is a sequence $\langle\delta_0,\dots ,\delta_k\rangle$ of ordinals with $\delta_0=\beta$, $\delta_k=\gamma$ and $\{\delta_i\}(n)=\delta_{i+1}$ for all $i<k$. The following properties are familiar from the usual fast-growing hierarchy:

\begin{lemma}\label{lem:fepsstar-hierarchy-basic}
For all numbers $m,n$ and ordinals $\alpha,\beta<\varepsilon_0$ the following holds:
\begin{enumerate}[label=(\roman*)]
\item We have $n\leq n^2<F_{\varepsilon_0+\alpha}^\diamond(n)$.
\item If $m\leq n$ then $F_{\varepsilon_0+\alpha}^\diamond(m)\leq F_{\varepsilon_0+\alpha}^\diamond(n)$.
\item If $\alpha\rightarrow_n\beta$ then $F_{\varepsilon_0+\beta}^\diamond(n)\leq F_{\varepsilon_0+\alpha}^\diamond(n)$.
\end{enumerate}
\end{lemma}
\begin{proof}
We repeat the well-known proof for the usual fast-growing hierarchy (see \cite[Proposition 2.5]{solovay81}), with minor modifications in the base case: Claim (i) is shown by induction on $\alpha$. For $\alpha=0$ we have
\begin{equation*}
n^2<F_{\omega_{\feps^{-1}(n)+1}}(n)=\feps^\diamond(n)
\end{equation*}
by \cite[Proposition 5.4]{sommer95}. Successor and limit case are easy. Claims (ii) and (iii) are shown by a simultaneous induction on $\alpha$. Concerning $\alpha=0$ it is easy to see that $m\leq n$ implies $\feps^{-1}(m)\leq\feps^{-1}(n)$. Then
\begin{equation*}
\feps^\diamond(m)=F_{\omega_{\feps^{-1}(m)+1}}(m)\leq F_{\omega_{\feps^{-1}(n)+1}}(m)\leq F_{\omega_{\feps^{-1}(n)+1}}(n)=\feps^\diamond(n)
\end{equation*}
follows by \cite[Lemma 2.3, Proposition 2.12]{rathjen13}. Claim (iii) is trivial for $\alpha=0$. In case $\alpha=\gamma+1$ claim (ii) holds by
\begin{multline*}
F_{\varepsilon_0+\alpha}^\diamond(m)=(F_{\varepsilon_0+\gamma}^\diamond)^{m+1}(m)\leq (F_{\varepsilon_0+\gamma}^\diamond)^{m+1}(n)\leq\\
\leq (F_{\varepsilon_0+\gamma}^\diamond)^{n+1}(n)=F_{\varepsilon_0+\alpha}^\diamond(n),
\end{multline*}
due to the induction hypothesis and claim (i). Concerning (iii) note that $\{\alpha\}(n)=\gamma$ forces $\beta=\gamma$ or $\gamma\rightarrow_n\beta$. Thus
\begin{equation*}
F_{\varepsilon_0+\beta}^\diamond(n)\leq F_{\varepsilon_0+\gamma}^\diamond(n)\leq (F_{\varepsilon_0+\gamma}^\diamond)^{n+1}(n)=F_{\varepsilon_0+\alpha}^\diamond(n)
\end{equation*}
follows by the induction hypothesis and claim (i). Let us come to (ii) for a limit ordinal $\alpha$: By \cite[Proposition 2.12]{rathjen13} we have $\{\alpha\}(n)\rightarrow_m \{\alpha\}(m)$. Then
\begin{equation*}
F_{\varepsilon_0+\alpha}^\diamond(m)= F_{\varepsilon_0+\{\alpha\}(m)}^\diamond(m)\leq F_{\varepsilon_0+\{\alpha\}(n)}^\diamond(m)\leq F_{\varepsilon_0+\{\alpha\}(n)}^\diamond(n)=F_{\varepsilon_0+\alpha}^\diamond(n)
\end{equation*}
uses the induction hypothesis of both (iii) and (ii). As for (iii), note that $\alpha\rightarrow_n\beta$ implies $\beta=\{\alpha\}(n)$ or $\{\alpha\}(n)\rightarrow_n\beta$. Thus
\begin{equation*}
F_{\varepsilon_0+\beta}^\diamond(n)\leq F_{\varepsilon_0+\{\alpha\}(n)}^\diamond(n)=F_{\varepsilon_0+\alpha}^\diamond(n)
\end{equation*}
follows from the induction hypothesis.
\end{proof}

To approach Theorem \ref{thm:total-functions-slow-reflection-mod} we bound the functions $F_{\varepsilon_0+\alpha}^\diamond$ in terms of the usual fast-growing hierarchy:

\begin{lemma}
Consider numbers $l,m,n$ with $m>0$ and an ordinal $\alpha\leq\omega_m$ which satisfy $(F_{\omega_m+\alpha})^l(n)\leq\feps(m)$. Then we have
\begin{equation*}
(F_{\varepsilon_0+\alpha}^\diamond)^l(n)\leq (F_{\omega_m+\alpha})^l(n).
\end{equation*}
\end{lemma}
\begin{proof}
We argue by transfinite induction on $\alpha$ with a side induction on $l$. The base $l=0$ of the side induction amounts to the trivial inequality $n\leq n$. So let us come to the side induction step $l\leadsto l+1$: There we have the assumption $(F_{\omega_m+\alpha})^{l+1}(n)\leq\feps(m)$. Abbreviating $N:=(F_{\varepsilon_0+\alpha}^\diamond)^l(n)$ our task is to show $F_{\varepsilon_0+\alpha}^\diamond(N)\leq (F_{\omega_m+\alpha})^{l+1}(n)$. We will use some well-known monotonicity properties of the fast-growing hierarchy, which can be found in \cite[Section 2]{rathjen13} (or \cite[Section 2]{solovay81}, with slightly different fundamental sequences). For example we have $(F_{\omega_m+\alpha})^l(n)<(F_{\omega_m+\alpha})^{l+1}(n)$, which allows us to apply the side induction hypothesis and obtain
\begin{equation*}
N\leq (F_{\omega_m+\alpha})^l(n)<\feps(m).
\end{equation*}
Now we distinguish the following cases:\\
\emph{Case $\alpha=0$:} Let us first show
\begin{equation}\label{eq:converse-feps-inverse}
\feps^{-1}(N)+1\leq m.
\end{equation}
Aiming at a contradiction, assume that we have $m\leq\feps^{-1}(N)$. Observe that this implies $\feps^{-1}(N)>0$. Invoking the definition of $\feps^{-1}$ we would then get
\begin{equation*}
\feps(m)\leq\feps(\feps^{-1}(N))\leq  N,
\end{equation*}
which contradicts $N<\feps(m)$ from above. Now in view of (\ref{eq:converse-feps-inverse}) we obtain
\begin{equation*}
\feps^\diamond(N)=F_{\omega_{\feps^{-1}(N)+1}}(N)\leq F_{\omega_m}(N)\leq (F_{\omega_m})^{l+1}(n),
\end{equation*}
which is the side induction step in the case $\alpha=0$.\\
\emph{Case $\alpha=\beta+1$:} First observe
\begin{equation*}
(F_{\omega_m+\beta})^{N+1}(N)=F_{\omega_m+\alpha}(N)\leq (F_{\omega_m+\alpha})^{l+1}(n)\leq\feps(m).
\end{equation*}
This allows us to apply the main induction hypothesis with $N,N+1$ and $\beta$ at the places of $n,l$ and $\alpha$, respectively. We get
\begin{multline*}
F_{\varepsilon_0+\alpha}^\diamond(N)=(F_{\varepsilon_0+\beta}^\diamond)^{N+1}(N)\leq(F_{\omega_m+\beta})^{N+1}(N)=\\
=F_{\omega_m+\alpha}(N)\leq (F_{\omega_m+\alpha})^{l+1}(n).
\end{multline*}
\emph{Case $\alpha$ limit:} The condition $\alpha\leq\omega_m$ implies $\omega_m+\{\alpha\}(N)=\{\omega_m+\alpha\}(N)$ (the ordinal $\omega_m$ meshes with $\alpha$, see \cite[Section 2]{rathjen13}). Then we have
\begin{equation*}
F_{\omega_m+\{\alpha\}(N)}(N)=F_{\omega_m+\alpha}(N)\leq (F_{\omega_m+\alpha})^{l+1}(n)\leq\feps(m).
\end{equation*}
Now apply the main induction hypothesis with $N,1$ and $\{\alpha\}(N)$ at the places of $n,l$ and $\alpha$, to get
\begin{multline*}
F_{\varepsilon_0+\alpha}^\diamond(N)=F_{\varepsilon_0+\{\alpha\}(N)}^\diamond(N)\leq F_{\omega_m+\{\alpha\}(N)}(N)=\\
=F_{\omega_m+\alpha}(N)\leq (F_{\omega_m+\alpha})^{l+1}(n).
\end{multline*}
We have thus completed the side induction step in all possible cases.
\end{proof}

From the lemma we can deduce the following result, which could be described as the ``combinatorial half" of Theorem \ref{thm:total-functions-slow-reflection-mod}:

\begin{proposition}\label{prop:modified-hierarchy-dominated-feps}
For each $\alpha<\varepsilon_0$ there is a number $N$ such that we have
\begin{equation*}
F_{\varepsilon_0+\alpha}^\diamond(\feps(n\dotminus 1))\leq\feps(n)\qquad\text{for all $n\geq N$}.
\end{equation*}
In particular $F_{\varepsilon_0+\alpha}^\diamond$ is eventually dominated by $\feps$.
\end{proposition}
\begin{proof}
Consider some $\alpha<\varepsilon_0$. We shall see that the proposition holds for any $N>0$ with $\omega_N\rightarrow_N\alpha+1$. Let us first show that such a number $N$ exists: As a first approximation take some $N_0>0$ with $\alpha<\omega_{N_0}$. From \cite[Lemma 2.6]{solovay81} we get a number $N$ with $\omega_{N_0}\rightarrow_N\alpha+1$, and by \cite[Corollary 2.4]{solovay81} we may assume $N\geq N_0$. By \cite[Proposition 2.12]{rathjen13} we have $\omega_N\rightarrow_N \omega_{N_0}$, and together this implies $\omega_N\rightarrow_N\alpha+1$ as desired. To verify the proposition consider an arbitrary number $n\geq N$. We would like to apply the previous lemma with $n,1$ and $\feps(n-1)$ at the places of $m,l$ and $n$, respectively. To do so we must verify the condition
\begin{equation}\label{eq:comparison-slow-hierarchy-auxiliary}
F_{\omega_n+\alpha}(\feps(n-1))\leq\feps(n)
\end{equation}
of the lemma. Using \cite[Lemma 2.7]{rathjen13} we get $\omega_n+\alpha\rightarrow_n\omega_n$, and then
\begin{multline*}
F_{\omega_n+\alpha}(\feps(n-1))=F_{\omega_n+\alpha}(F_{\omega_n}(n-1))\leq F_{\omega_n+\alpha}(F_{\omega_n}(n))\leq\\
\leq (F_{\omega_n+\alpha})^2(n)\leq (F_{\omega_n+\alpha})^{n+1}(n)=F_{\omega_n+\alpha+1}(n).
\end{multline*}
The next step is to show $\omega_{n+1}\rightarrow_n\omega_n+\alpha+1$: From \cite[Lemma 2.10, 2.13]{rathjen13} we get $\omega_{n+1}\rightarrow_n\omega^{\omega_{n-1}+1}$. In view of $\{\omega^{\omega_{n-1}+1}\}(1)=\omega_n+\omega_n$ we can use \cite[Proposition 2.12]{rathjen13} to obtain $\omega_{n+1}\rightarrow_n\omega_n+\omega_n$. Since $\omega_n$ meshes with $\omega_n$ it only remains to show $\omega_n\rightarrow_n\alpha+1$. This follows from the above $\omega_N\rightarrow_N\alpha+1$ using \cite[Corollary 2.4]{solovay81} and \cite[Proposition 2.12]{rathjen13}. Now we get
\begin{equation*}
F_{\omega_n+\alpha+1}(n)\leq F_{\omega_{n+1}}(n)=\feps(n),
\end{equation*}
which completes the proof of (\ref{eq:comparison-slow-hierarchy-auxiliary}). This allows us to apply the previous lemma, and we finally obtain
\begin{equation*}
F_{\varepsilon_0+\alpha}^\diamond(\feps(n-1))\leq F_{\omega_n+\alpha}(\feps(n-1))\leq\feps(n).
\end{equation*}
To deduce that $F_{\varepsilon_0+\alpha}^\diamond$ is eventually dominated by $\feps$ use $n\leq\feps(n-1)$ and the fact that $F_{\varepsilon_0+\alpha}^\diamond$ is monotone.
\end{proof}

The previous proposition is complemented by the following result:

\begin{proposition}
Any provably total function of $\pa+\drfn$ is eventually dominated by one of the functions $F_{\varepsilon_0+\alpha}^\diamond$ with $\alpha<\varepsilon_0$.
\end{proposition}
\begin{proof}
By Proposition \ref{prop:slow-reflection-unary-function} the slow reflection principle $\drfn$ is equivalent to the statement that the function $\feps^\diamond$ is total. It is a classical result that any provably total function of Peano Arithmetic is eventually dominated by one of the functions $F_\alpha$ with $\alpha<\varepsilon_0$ from the fast-growing hierarchy. We need to see that this remains valid when one adds the base function $\feps^\diamond$ (both as an axiom and as initial function of the fast-growing hierarchy). Indeed a general result to this effect is shown as part of the proof of \cite[Theorem 16]{kristiansen12}. However, in \cite{kristiansen12} the approach to the fast-growing hierarchy is somewhat different: The paper works with norms of ordinals rather than explicit fundamental sequences. This appears to be a technicality, but rather than working out a detailed comparison we take the more direct way and reprove \cite[Theorem 16]{kristiansen12} in our setting:\\
As basis for our proof we take the analysis of the provably total functions of Peano Arithmetic in \cite{buchholz-wainer-87}. We assume that the reader has access to this paper. Note that the notation $\beta<_k\alpha$ in \cite{buchholz-wainer-87} refers to the same ``step down"-relation that we write as $\alpha\rightarrow_k\beta$. First of all we need to extend the formalization of Peano Arithmetic in \cite[Section 2]{buchholz-wainer-87} by the axiom $\feps^\diamond\!\downarrow$. To do so, recall that the graph of $\feps^\diamond$ is defined by a $\Delta_0$-formula and is thus elementary. So the formal system of \cite[Section 2]{buchholz-wainer-87} already contains a relation symbol $\feps^\diamond(\cdot)=\cdot$ and defining axioms corresponding to its elementary definition. Using this relation symbol we extend the formal system by the new axiom $\forall_x\exists_y\feps^\diamond(x)=y$. Next, we need to adapt the infinitary proof system of \cite[Section 3]{buchholz-wainer-87}. This system contains a special relation symbol $\cdot\in\n$ which will be interpreted as a finite approximation to the set of natural numbers. The infinitary system contains an axiom which places zero in $\n$ and a rule which allows us to put in successors: 
\begin{equation*}
(\n)\qquad\text{if}\quad\vdash^\alpha\Gamma,n\in\n\quad\text{then}\quad\vdash^{\alpha+1}\Gamma,n+1\in\n.
\end{equation*}
We need to add a new rule which gives access to values of the function $\feps^\diamond$ (it is important that the increase in the ordinal bound is independent of $n$):
\begin{equation*}
(\dn)\qquad\text{if}\quad\vdash^\alpha\Gamma,n\in\n\quad\text{then}\quad\vdash^{\alpha+1}\Gamma,\feps^\diamond(n)\in\n.
\end{equation*}
Using this rule the embedding lemma is easily extended by a proof of the axiom $\forall_x\exists_y\feps^\diamond(x)=y$ in the infinite system: Since the prime formula $\feps^\diamond(n)=\overline{\feps^\diamond(n)}$ is true we get $\vdash^1 n\notin\n,\feps^\diamond(n)=\overline{\feps^\diamond(n)}$ for each $n$. The axiom $\vdash^0 n\notin\n,n\in\n$ and the new rule $(\dn)$ yield $\vdash^1 n\notin\n,\feps^\diamond(n)\in\n$. Introducing a conjunction and an existential quantifier we obtain $\vdash^3 n\notin\n,\exists_{y\in\n}\feps^\diamond(n)=y$. To keep the coefficients in the ordinal bound small we now apply accumulation: In view of $\omega\rightarrow_2 3$ we can conclude $\vdash^\omega n\notin\n,\exists_{y\in\n}\feps^\diamond(n)=y$. By disjunction introduction and the $\omega$-rule we arrive at $\vdash^{\omega+3} \forall_{x\in\n}\exists_{y\in\n}\feps^\diamond(x)=y$. Using accumulation again we get
\begin{equation*}
\vdash^{\omega\cdot 2} \forall_{x\in\n}\exists_{y\in\n}\feps^\diamond(x)=y,
\end{equation*}
precisely as needed for the extended embedding lemma. It is straightforward to check that inversion, reduction and cut-elimination remain valid: In this respect the new rule $(\dn)$ behaves just as the original rule $(\n)$. In the bounding lemma the bound $F_\alpha(k)$ is replaced by $F_{\varepsilon_0+\alpha}^\diamond(k)$:
\begin{equation*}
\parbox{0.85\textwidth}{Assume that we have $\vdash^\alpha n_1\notin\n,\dots ,n_r\notin\n,\Gamma$ with cut rank 0, where $\Gamma$ only contains closed positive $\Sigma_1(\n)$-formulas. Then $\Gamma$ is true in $F_{\varepsilon_0+\alpha}^\diamond(k)$ for $k=\max(\{2\}\cup\{3n_1,\dots ,3n_r\})$.}
\end{equation*}
Recall that positive $\Sigma_1(\n)$-formulas only contain the connectives $\lor,\land,\exists$ and do not contain subformulas of the form $n\notin\n$. A closed sequent is called true in $m$ if the disjunction of its formulas is true under the interpretation $\n=\{n\,|\,3n<m\}$ of the special relation symbol. To prove the bounding lemma one argues by induction on $\alpha$ and distinguishes cases according to the last rule of the deduction $\vdash^\alpha n_1\notin\n,\dots ,n_r\notin\n,\Gamma$. Using Lemma \ref{lem:fepsstar-hierarchy-basic} this is straightforward and essentially as in \cite{buchholz-wainer-87}. Let us only consider the case of a deduction that ends in the new rule $(\dn)$: Then $\Gamma$ contains a formula of the form $\feps^\diamond(n)\in\n$ and we have $\vdash^\beta n_1\notin\n,\dots ,n_r\notin\n,\Gamma,n\in\n$ with $\alpha=\beta+1$. Since the premise of the rule contains no new formula of the form $m\notin\n$ the number $k$ is unchanged and the induction hypothesis tells us that $\Gamma,n\in\n$ is true in $F_{\varepsilon_0+\beta}^\diamond(k)$. There are two possibilities: If $\Gamma$ is true in $F_{\varepsilon_0+\beta}^\diamond(k)$ then it is also true in $F_{\varepsilon_0+\alpha}^\diamond(k)\geq F_{\varepsilon_0+\beta}^\diamond(k)$. Otherwise the formula $n\in\n$ must be true in $F_{\varepsilon_0+\beta}^\diamond(k)$, which means that we have $n\leq 3n<F_{\varepsilon_0+\beta}^\diamond(k)$. Using Lemma \ref{lem:fepsstar-hierarchy-basic} we observe $3\leq (F_{\varepsilon_0+\beta}^\diamond)^2(k)$ and infer
\begin{multline*}
3\cdot\feps^\diamond(n)\leq 3\cdot\feps^\diamond(F_{\varepsilon_0+\beta}^\diamond(k))\leq (F_{\varepsilon_0+\beta}^\diamond)^2(k)\cdot(F_{\varepsilon_0+\beta}^\diamond)^2(k)<\\
<(F_{\varepsilon_0+\beta}^\diamond)^3(k)\leq (F_{\varepsilon_0+\beta}^\diamond)^{k+1}(k)=F_{\varepsilon_0+\alpha}^\diamond(k).
\end{multline*}
This means that $\Gamma$ contains the formula $\feps^\diamond(n)\in\n$ which is true in $F_{\varepsilon_0+\alpha}^\diamond(k)$.\\
Now we can deduce the desired result as usual: Let $g$ be a provably total function of the theory $\pa+\feps^\diamond\!\downarrow$. From the given definition of $g$ we can read off an elementary relation $\chi_g$ such that we have
\begin{gather*}
g(m)=n\qquad\Leftrightarrow\qquad\mathbb N\vDash\exists_z\chi_g(m,n,z),\\
\pa+\feps^\diamond\!\downarrow\,\vdash\,\forall_x\exists_{y,z}\chi_g(x,y,z).
\end{gather*}
By embedding and cut elimination we get an ordinal $\alpha<\varepsilon_0$ and an infinitary deduction $\vdash^\alpha\forall_{x\in N}\exists_{y\in N}\exists_{z\in N}\chi_g(x,y,z)$ of cut rank $0$. Inversion yields a deduction
\begin{equation*}
\vdash^\alpha m\notin\n,\exists_{y\in N}\exists_{z\in N}\chi_g(m,y,z)
\end{equation*}
for each number $m$. Assume $m\geq 3$. By the bounding lemma there are numbers $n,k<F_{\varepsilon_0+\alpha}^\diamond(3m)$ such that $\chi_g(m,n,k)$ is true. Using Lemma \ref{lem:fepsstar-hierarchy-basic} we get
\begin{equation*}
g(m)<F_{\varepsilon_0+\alpha}^\diamond(3m)\leq F_{\varepsilon_0+\alpha}^\diamond(m^2)\leq (F_{\varepsilon_0+\alpha}^\diamond)^2(m)\leq F_{\varepsilon_0+\alpha+1}^\diamond(m),
\end{equation*}
which shows that $g$ is eventually dominated by $F_{\varepsilon_0+\alpha+1}^\diamond$.
\end{proof}

Putting pieces together we can deduce the main result of this section:

\begin{theorem}\label{thm:total-functions-slow-reflection-mod}
 For any provably total function $g$ of $\pa+\drfn$ there is a number $N$ such that we have
\begin{equation*}
 g(\feps(n\dotminus 1))\leq \feps(n)\qquad\text{for all $n\geq N$}.
\end{equation*}
In particular any provably total function of the theory $\pa+\drfn$ is eventually dominated by $\feps$.
\end{theorem}
\begin{proof}
Consider a function $g$ which is provably total in $\pa+\drfn$. The previous proposition provides an ordinal $\alpha<\varepsilon_0$ and a bound $N$ such that we have
\begin{equation*}
g(m)\leq F_{\varepsilon_0+\alpha}^\diamond(m)\quad\text{for all $m\geq N$}.
\end{equation*}
Increasing $N$ if necessary Proposition \ref{prop:modified-hierarchy-dominated-feps} yields
\begin{equation*}
F_{\varepsilon_0+\alpha}^\diamond(\feps(n\dotminus 1))\leq\feps(n)\quad\text{for all $n\geq N$}.
\end{equation*}
For $n\geq N$ we have $\feps(n\dotminus 1)\geq n\geq N$ and thus
\begin{equation*}
g(\feps(n\dotminus 1))\leq F_{\varepsilon_0+\alpha}^\diamond(\feps(n\dotminus 1))\leq\feps(n)
\end{equation*}
and
\begin{equation*}
g(n)\leq F_{\varepsilon_0+\alpha}^\diamond(n)\leq F_{\varepsilon_0+\alpha}^\diamond(\feps(n\dotminus 1))\leq\feps(n),
\end{equation*}
as required for Theorem \ref{thm:total-functions-slow-reflection-mod}.
\end{proof}

We remark that the theorem implies
\begin{equation*}
 \pa+\drfn\nvdash\rfn(\pa),
\end{equation*}
because the equivalence $\rfn(\pa)\leftrightarrow\feps\!\downarrow$ is provable in Peano Arithmetic (even in $\isigma_1$, as implied by Proposition \ref{prop:reflection-fragments-fast-growing}). The analogous result for slow consistency has been proved in \cite{rathjen13} (see also statement (\ref{eq:slow-con-weaker-con}) in Section \ref{sect:no-short-proofs-ph}). In \cite{freund-slow-reflection} we investigate the consistency strength of slow reflection (also for reflection formulas of complexity above $\Sigma_1$): In particular it is shown that $\pa+\drfn$ does not even prove the consistency of Peano Arithmetic. Further results on slow provability can be found in work of Henk and Pakhomov \cite{henk-pakhomov}. To conclude this paper, let us rephrase our computational analysis in terms of subrecursive degree theory (see \cite{kristiansen12}):

\begin{corollary}
The honest $\varepsilon_0$-elementary degree of $\feps^\diamond$ is a non-zero degree strictly below the degree of $\feps$.
\end{corollary}
\begin{proof}
First we must verify that $\feps^\diamond$ and $\feps$ are honest functions (in the sense of \cite{kristiansen12}). We already know that the two functions are monotone and have elementary graphs (since they can be defined by $\Delta_0$-formulas). It remains to show that they dominate the function $n\mapsto 2^n$: By straightforward computations we see that $F_2(n)\geq 2^n$ holds for all $n$. Since $\omega_{m+1}\rightarrow_n 2$ holds for $n\geq 1$ and any $m$ we obtain
\begin{equation*}
\feps^\diamond(n)=F_{\omega_{\feps^{-1}(n)+1}}(n)\geq F_2(n)\geq 2^n\qquad\text{for $n\geq 1$}.
\end{equation*}
In the separate case $n=0$ the inequality $\feps^\diamond(0)\geq 2^0$ is immediate. Similarly one shows that $\feps(n)\geq 2^n$ holds for all $n$. Now let us argue that the $\varepsilon_0$-elementary degree of $\feps^\diamond$ is non-zero: In the discussion just after Proposition \ref{prop:slow-reflection-unary-function} above we have seen that $\feps^\diamond$ eventually dominates any provably total function of Peano Arithmetic. Thus Peano Arithmetic cannot prove the totality of any honest representation of $\feps^\diamond$. In the notation of \cite[Section 10]{kristiansen12} this means that we do not have $\feps^\diamond\leq_\pa\mathbf 0$. Then \cite[Theorem 16]{kristiansen12} tells us that we cannot have $\feps^\diamond\leq_{\epsilon_0 E}\mathbf 0$. In other words, the $\varepsilon_0$-elementary degree of $\feps^\diamond$ is non-zero. Similarly Theorem \ref{thm:total-functions-slow-reflection-mod} tells us that $\feps\leq_\pa\feps^\diamond$ must fail, so that $\feps\leq_{\varepsilon_0 E}\feps^\diamond$ must fail as well. Thus $\feps^\diamond$ and $\feps$ do not have the same $\varepsilon_0$-elementary degree. On the other hand it is easy to see $\feps^{\diamond}\leq_{\varepsilon_0 E}\feps$: Since $\feps^\diamond$ has an elementary graph and is dominated by $\feps$ it is even elementary in $\feps$, by bounded minimization.\\
We remark that the use of \cite[Theorem 16]{kristiansen12} is, in some sense, a detour: Rather than considering provability in Peano Arithmetic one could use the ``Generalized Growth Theorem" \cite[Theorem 13]{kristiansen12} in combination with Proposition \ref{prop:modified-hierarchy-dominated-feps} above. Then, however, one has the technical task to reconcile the different definitions of the fast-growing hierarchy in our paper and in \cite{kristiansen12}.
\end{proof}

\section*{Acknowledgements}

\noindent I am very grateful to Michael Rathjen, my Ph.D.\ supervisor, for his advise and guidance. I also want to thank the referee for his helpful comments, which particularly improved Section \ref{sect:provably-total-slow-reflection} of the paper.

\bibliographystyle{alpha}
\bibliography{Proof-Length-Paris-Harrington_Freund}

\end{document}